\documentclass{amsart}
\usepackage[ocgcolorlinks,linktoc=all]{hyperref}
\hypersetup{citecolor=blue,linkcolor=red}
\newtheorem{theorem}{Theorem}[section]
\newtheorem*{thm}{Theorem}
\newtheorem*{thmA}{Theorem A}

\newtheorem*{thmB}{Theorem B}

\newtheorem*{thmmain}{Theorem}
\newtheorem{lemma}[theorem]{Lemma}
\newtheorem{proposition}[theorem]{Proposition}

\newtheorem{corollary}[theorem]{Corollary}
\theoremstyle{definition}

\theoremstyle{remark}
\newtheorem{remark}[theorem]{Remark}
\numberwithin{equation}{section}

\begin{document}
\title[The planar Busemann-Petty centroid inequality and its stability]
 {The planar Busemann-Petty centroid inequality and its stability}
\author[M.N. Ivaki]{Mohammad N. Ivaki}
\address{Institut f\"{u}r Diskrete Mathematik und Geometrie, Technische Universit\"{a}t Wien,
Wiedner Hauptstr. 8--10, 1040 Wien, Austria}
\curraddr{}
\email{mohammad.ivaki@tuwien.ac.at}


\begin{abstract}
In [Centro-affine invariants for smooth convex bodies, Int. Math. Res. Notices. doi: 10.1093/imrn/rnr110, 2011] Stancu introduced a family of centro-affine normal flows, $p$-flow, for $1\leq p<\infty.$ Here we investigate the asymptotic behavior of the planar $p$-flow for $p=\infty$ in the class of smooth, origin-symmetric convex bodies. First, we prove that the $\infty$-flow evolves suitably normalized origin-symmetric solutions to the unit disk in the Hausdorff metric, modulo $SL(2).$ Second, using the $\infty$-flow and a Harnack estimate for this flow, we prove a stability version of the planar Busemann-Petty centroid inequality in the Banach-Mazur distance. Third, we prove that the convergence of normalized solutions in the Hausdorff metric can be improved to convergence in the $\mathcal{C}^{\infty}$ topology.
\end{abstract}

\maketitle
\section{Introduction}
The setting of this paper is $n$-dimensional Euclidean space. A compact convex subset of $\mathbb{R}^{n}$ with non-empty interior is called a \emph{convex body}. The set of convex bodies in $\mathbb{R}^{n}$ is denoted by $\mathcal{K}^n$. Write $\mathcal{K}^n_{e}$ for the set of origin-symmetric convex bodies and $\mathcal{K}^n_{0}$ for the set of convex bodies whose interiors contain the origin. Also write respectively $\mathcal{F}^n$, $\mathcal{F}^n_0$, and $\mathcal{F}^n_e$  for the set of smooth ($\mathcal{C}^{\infty}$-smooth), strictly convex bodies in $\mathcal{K}^n$, $\mathcal{K}^n_{0}$, and $\mathcal{K}^n_{e}$.

The support function of $K\in \mathcal{K}^n$, $h_K:\mathbb{S}^{n-1}\to\mathbb{R}$, is defined by \[h_{K}(u)=\max_{x\in \partial K}\langle x,u\rangle,\]
where $\langle \cdot,\cdot\rangle$ stands for the usual inner product of $\mathbb{R}^n.$
For $K\in\mathcal{K}^n$, write $V(K)$ for its Lebesgue measure as a subset of $\mathbb{R}^{n}$.
The centroid body $\Gamma K\in \mathcal{K}_e^n$ of convex body $K$ is the convex body whose support function is given by
\begin{equation}\label{e: support centroid}
h_{\Gamma K}(u)=\frac{1}{V(K)}\int_{K}|\langle u,x\rangle|dx.
\end{equation}
It was proved by Petty that $\Gamma \Phi  K=\Phi \Gamma K$ for $\Phi \in GL(n)$, \cite{P} (for newer references, see also \cite[Theorem 9.1.3]{Gardner} and \cite[Lemma 2.6]{LYZ2}). Moreover, a very interesting theorem of Petty states that the centroid body of a convex body is always of class $\mathcal{C}^{2}_+$ (the class of convex bodies with two times continuously differentiable boundary hypersurface and positive principal curvatures everywhere), \cite{P}.
Geometrically, for an origin-symmetric convex body $K$, the boundary of $\Gamma K$ is the locus of the centroids of all the halves of $K$ obtained by cutting $K$ with hyperplanes through the origin.
This concept dates back at least to Dupin. The Busemann-Petty centroid inequality, which was conjectured by Blaschke \cite{BW} and established by Petty in his seminal work \cite{P} using the Busemann random simplex inequality, states that
\begin{equation}\label{e: ineq centroid}
\frac{V(\Gamma K)}{V(K)}\geq \left(\frac{2\omega_{n-1}}{(n+1)\omega_n}\right)^n,
\end{equation}
with equality if and only if $K$ is an origin centered ellipsoid. See important extensions of this inequality by Campi, Gronchi, Haberl, Schuster, Lutwak, Yang, and Zhang in the context of the $L_p$ Brunn-Minkowski theory and the Orlicz Brunn-Minkowski theory \cite{CG1,CG2,HS,li2011new,Lutwak4,LYZ1,LYZ2}.

Let $K$ be a body in $\mathcal{F}^n_0$, whose boundary $\partial K$ is smoothly embedded in $\mathbb{R}^n$ by
$$X_K:\partial K\to\mathbb{R}^n.$$
We write $\nu:\partial K\to \mathbb{S}^{n-1}$ for the Gauss map of $\partial K$. That is, at each point $x\in\partial K$, $\nu(x)$ is the unit outwards normal at $x$.
The support function of $\partial K$ can also be described as
$$h_{K}(u):=\langle X_K(\nu^{-1}(u)),u\rangle.$$

The standard metric on $\mathbb{S}^{n-1}$ is denoted by $\bar{g}_{ij}$ and its standard Levi-Civita connection is denoted by $\bar{\nabla}$.
Write $\mathcal{G}_K$ for the Gauss curvature of $\partial K$, and $\mathcal{S}_K$ for the reciprocal Gauss curvature, as a function on the unit sphere. They are related to the support function by
\[\frac{1}{\mathcal{G}_K(\nu^{-1})}=\mathcal{S}_K:=\det_{\bar{g}}(\bar{\nabla}_i\bar{\nabla}_jh_K+\bar{g}_{ij}h_K).\]
When there are not any ambiguities we will drop the sub-scripts $K$ and $K_t.$

The flow whose definition will be given, was first introduced by Stancu in \cite{S}. In \cite{S} Stancu introduces new equicentro-affine invariants, and she provides a geometric interpretation of the $L_{\Phi}$ affine surface area introduced by Ludwig and Reitzner \cite{LR}. Further applications are given by Stancu in \cite{S4} in connection to the Paouris-Werner invariant defined on convex bodies \cite{PW}.

Let $K\in \mathcal{F}_{0}^n$. A family of convex bodies $\{K_t\}_t\in \mathcal{K}_{0}^n$  given by the smooth map $X:\partial K\times[0,T)\to \mathbb{R}^n$ is said to be a solution to the $p$-flow if $X$ satisfies the initial value problem
\begin{equation}\label{e: flow0}
 \partial_{t}X(x,t)=-\left(\frac{\mathcal{G}(x,t)}{\langle X(x,t), \nu(x,t)\rangle^{n+1}}\right)^{\frac{p}{p+n}-\frac{1}{n+1}}\mathcal{G}^{\frac{1}{n+1}}(x,t)\, \nu(x,t),~~
 X(\cdot,0)=X_K,
\end{equation}
for a fixed $1< p<\infty$. In this equation, $0<T<\infty$ is the maximal time that the solution exists, and $\nu(x,t)$ is the unit normal to the hypersurface $ X(\partial K,t)=\partial K_t$ at $X(x,t).$

The long time behavior of the flow in $\mathbb{R}^n$ with $K_0\in \mathcal{F}_{e}^n$ was investigated in \cite{Ivaki,IS,S6}. It was proved that for $1<p<\frac{n}{n-2}$ the volume-preserving $p$-flow, which keeps the volume of the evolving bodies ﬁxed and equal to the volume of the unit ball, evolves each body in $ \mathcal{F}_{e}^{n}$ to the unit ball in $\mathcal{C}^{\infty}$, modulo $SL(n).$ Two applications arising from the tools developed in \cite{Ivaki} to the $L_{-2}$ Minkowski problem and to the stability of the $p$-affine isoperimetric inequality in $\mathbb{R}^2$ were given in \cite{Ivaki1,Ivaki2}. The case $p=1$, corresponds to the well-known affine normal flow. This case in dimension two was addressed by Sapiro and Tannenbaum \cite{ST} and in higher dimensions by Andrews \cite{BA,BA2}. It was proved by Andrews that the volume-preserving affine normal flow evolves any convex initial bounded open set exponentially fast in the $\mathcal{C}^{\infty}$ topology to an ellipsoid. Ancient solutions, existence and regularity of solutions to the affine normal flow on non-compact strictly convex hypersurfaces have been treated in \cite{LT} by Loftin and Tsui.

It is easy to see from the definition of the support function that as convex bodies $K_t$ evolve by (\ref{e: flow0}), their corresponding support functions satisfy the partial differential equation
\begin{equation}\label{e: flow1}
\partial_t h(u,t)=-h(u,t)\left(\frac{1}{\mathcal{S}h^{n+1}}\right)^{\frac{p}{p+n}}(u,t),~ h(\cdot,0)=h_{K_0}(\cdot),
\end{equation}
see also \cite{S}.
The short time existence and uniqueness of solutions for smooth and strictly convex initial bodies follow from the strict parabolicity of the equation and were established in \cite{S}.

In this paper, we employ flow (\ref{e: flow1}) with $K_0\in\mathcal{F}_e^2$ and for the case $p=\infty$ and $n=2$:
\begin{align}\label{e: asli}
\left\{
  \begin{array}{ll}
    \partial_t h(u,t)=-\frac{1}{h^2\mathcal{S}}(u,t)=-\frac{1}{h^2(u,t)\left(h_{\theta\theta}+h\right)(u,t)}, \\
    h(\cdot,0)=h_{K_0}(\cdot).
  \end{array}
\right
.
\end{align}
Notice that the solution to (\ref{e: asli}) remains origin-symmetric. Here and then, we identify the unit normal $u=(\cos \theta, \sin \theta)$ with $\theta$.
The short time existence and uniqueness of solutions for a smooth and strictly convex initial body follow from the strict parabolicity of the equation.

Write $B$ for the unit disk of $\mathbb{R}^2.$ The first main result of the paper is contained in the following theorem.
\begin{thmA}
Let $K_0\in\mathcal{F}_{e}^2.$ Then there exists a unique solution $X:\mathbb{S}^1\times [0,T)\to\mathbb{R}^2$ of flow (\ref{e: asli}) with initial data $X_{K_0}$. The solution remains smooth and strictly convex on $[0,T)$ for a finite time $T>0$. The rescaled convex bodies given by $\left(V(B)/V(K_{t})\right)^{1/2}K_{t}$ converge sequentially in $\mathcal{C}^{\infty}$ to the unit disk, modulo $SL(2)$, as $t\rightarrow T.$
\end{thmA}
It is, probably, worth mentioning that the method to conclude the long time behavior in this paper is significantly different from the method of \cite{Ivaki} and in particular relies on a number of affinely associated bodies from convex geometry, namely $K^{\ast}$, $\Pi K$, and $\Lambda K$ (See the next section for the definitions of these bodies.). Essentially, monotonicity of $V(\Gamma K_t)/V(K_t)$ along flow (\ref{e: asli}) plays a key role. A byproduct of the monotonicity is a stability version of planar Busemann-Petty centroid inequality.

Consider an inequality $\mathcal{I}:\mathcal{K}^n\to \mathbb{R}$, $\mathcal{I}(K)\geq 0$, for which the equality is precisely obtained for a family $\mathcal{M}\subseteq\mathcal{K}^n$. If for a convex body $L$ and some $\varepsilon>0$ there holds $\mathcal{I}(L)\leq \varepsilon$, what can be said about the distance of $L$, in an appropriate distance, from the objects in $\mathcal{M}$? Questions of this type investigate the stability of geometric inequalities and have appeared in the work of Minkowski and Bonnesen. See the beautiful survey \cite{Groemer} of Groemer for a wealth of information and references.

In recent times stability of several significant inequalities has been addressed which most of these geometric inequalities have  balls, ellipsoids, or simplices as objects for occurrence of the equality. To give examples, we mention  stability versions of the Brunn-Minkowski inequality due to Diskant \cite{Dis}, and due to A. Figalli, F. Maggi and A. Pratelli \cite{FMP}, stability of the Orlicz-Petty projection inequality \cite{Bor}, stability of the Rogers-Shephard inequality \cite{B1}, stability of the Blaschke-Santal\'{o} inequality, and the affine isoperimetric inequality in $\mathbb{R}^n$ \cite{B} by B\"{o}r\"{o}czky, stability of the reverse Blaschke-Santal\'{o} inequality by B\"{o}r\"{o}czky and Hug \cite{BH}, stability of the Pr\'{e}kopa-Leindler inequality by Ball and B\"{o}r\"{o}czky \cite{BB,BB2}, and more recently, stability of the functional forms of the Blaschke-Santal\'{o} inequality by Barthe, B\"{o}r\"{o}czky and Fradelizi \cite{BBF}. The second aim of this paper is to prove a stability version of the planar Busemann-Petty centroid inequality using (\ref{e: asli}). Within the last few years, a substantial amount of research was devoted to investigate applications of geometric flows to different areas of mathematics. In particular, there are several major contributions of geometric flows to convex geometry: a proof of the affine isoperimetric inequality by Andrews using the affine normal flow \cite{BA}, obtaining the necessary and sufficient conditions for the existence of a solution to the discrete $L_0$-Minkowski problem using crystalline curvature flow by Stancu \cite{S0,S5,S3} and independently by Andrews \cite{BA5}, an application of the affine normal flow to the regularity of minimizers of Mahler volume by Stancu \cite{S1}, and obtaining quermassintegral inequalities for $k$-convex star-shaped domains using a family of expanding flows \cite{PJ1}.

To state our stability result, we recall the Banach-Mazur distance. A natural tool in affine geometry of convex bodies is the Banach-Mazur distance which for two convex bodies $K,\bar{K}\in \mathcal{K}^n$ is defined by
\[d_{\mathcal{BM}}(K,\bar{K})=\min\{ \lambda\geq 1: K-x\subseteq \Phi (\bar{K}-y)\subseteq \lambda (K-x),~\Phi\in GL(n),~x,y\in\mathbb{R}^n\}.\]
For origin-symmetric bodies we can assume $x=y$ coincides with the origin.
It is easy to see that $d_{\mathcal{BM}}(K, \Phi \bar{K})=d_{\mathcal{BM}}(K,\bar{K})$ for all $\Phi\in GL(n).$ Moreover, the Banach-Mazur distance is multiplicative. That is,  for $K_1,K_2,K_3\in \mathcal{K}_e^n$ the following inequality holds:
\[d_{\mathcal{BM}}(K_1,K_3)\leq d_{\mathcal{BM}}(K_1,K_2)d_{\mathcal{BM}}(K_2,K_3).\]

The second main result of the paper is contained in the following theorem.
\begin{thmB}
There exist $\varepsilon_0>0$,\ and $\gamma>0$ such that the following holds:
If $0<\varepsilon<\varepsilon_0$ and $K$ is a convex body in $\mathbb{R}^2$ such that $\frac{V(\Gamma K)}{V(K)}\leq \left(\frac{4}{3\pi}\right)^2(1+\varepsilon)$, then $d_{\mathcal{BM}}(K,B)\leq 1+\gamma\varepsilon^{1/8}.$
Furthermore, if $K$ is origin-symmetric, then
$d_{\mathcal{BM}}(K,B)\leq 1+\gamma\varepsilon^{1/4}.$
\end{thmB}
Since $\mathcal{I}:\mathcal{K}^2\to \mathbb{R}$ defined by $\mathcal{I}(K):=\frac{V(\Gamma K)}{V(K)}-\left(\frac{4}{3\pi}\right)^2$ is a continuous functional in the Hausdorff distance, it suffices to prove Theorem B for bodies in $\mathcal{F}^2$. Moreover, in light of a theorem of Campi and Gronchi \cite{CG1} which states the volume of the centroid body is not increased after a Steiner symmetrization, and Theorem 1.4 of B\"{o}r\"{o}czky \cite{B}, it is enough to first prove Theorem B for bodies in $\mathcal{F}_e^2.$ The \emph{idea} to prove this result is as follows: Let $K\in\mathcal{F}_e^2$ satisfies the assumption of Theorem B and let $\{K_t\}$ be the solution to flow (\ref{e: asli}) with $K_0=\Phi K$, for an appropriate $\Phi\in GL(2)$. It will be proved that $V(\Gamma K_t)/V(K_t)$ is non-increasing in time. Furthermore, calculating the evolution equation of $V(\Gamma K_t)/V(K_t)$ we prove that its time derivative is controlled by a stable area ratio which is zero only for ellipses. From this observation, we will conclude that for some time $s>0$ and close to zero, $K_s$ must be close to the unit disk in the Banach-Mazur distance. Additionally, we can also control the distance between $K$ and $K_{s}$ in the Banach-Mazur distance (using a Harnack estimate) provided that $\varepsilon$ is small enough. Putting these observations altogether, we are able to prove that $K_0$ is close to the unit disk in the Banach-Mazur distance and so is $K.$ This approach to the stability problems was employed by the author to obtain the stability of the $p$-affine isoperimetric inequality for bodies in $\mathcal{K}_e^2$, \cite{Ivaki2}.

The paper is structured as follows: In the next section we recall some definitions and results from convex geometry. Section \ref{sec: basic prop} focuses on establishing basic properties of (\ref{e: asli}). We show that evolving bodies remain smooth, strictly convex and the area of the evolving convex bodies converge to zero in finite time. In Section \ref{sec: long time} we study the long time behavior of (\ref{e: asli}). To study the convergence of solutions, we resort to the evolution equation of $V(\Gamma K_t)/V(K_t)$ along the flow. The crucial result is that $V(\Gamma K_t)/V(K_t)$ is non-increasing along the flow. This observation implies that $\left(V(B)/V(K_t)\right)^{1/2}K_t$ converge in the Banach-Mazur distance to a limiting shape $\bar{K}_{\infty}$ with the property $\Lambda \bar{K}_{\infty}=\bar{K}_{\infty}$. It is here where using a theorem of Petty \cite{P1} on the latter equality in dimension two, we conclude the convergence of solutions to the unit disk modulo $SL(2)$. In Section \ref{sec: stability} we prove Theorem B. In the final section, we prove the sequential convergence of the normalized solution in $\mathcal{C}^{\infty}$ to the unit disk, modulo $SL(2).$

\textbf{Acknowledgment.}\\
For her many valuable suggestions, I am obliged to Monika Ludwig. I am indebted to the referee whose comments have led to improvements of this article. The work of the author was supported in part by Austrian Science Fund (FWF) Project P25515-N25.
\section{background and notation}
If $K,L$ are convex bodies and $0<a<\infty$, then the Minkowski sum $K+aL$ is defined by $h_{K+aL}=h_K+ah_L$ and the mixed volume $V_1(K,L)$ ($V(K,L)$ for planar convex bodies) of $K$ and $L$ is defined by
\[V_1(K,L)=\frac{1}{n}\lim_{a\to0^{+}}\frac{V(K+aL)-V(K)}{a}.\]
A fundamental fact is that corresponding to each convex body $K$, there is a unique Borel measure $S_K$ on the unit sphere such that
\[V_1(K,L)=\frac{1}{n}\int_{\mathbb{S}^{n-1}}h_L(u)dS_K(u)\]
for each convex body $L$. The measure $S_K$ is called the surface area measure of $K.$ Recall that if $K$ is $\mathcal{C}^2_+$, then $S_K$ is absolutely continuous with respect to $\sigma$, and the Radon-Nikodym derivative $dS_K(u)/d\sigma(u)$ defined on $\mathbb{S}^{n-1}$ is the reciprocal Gauss curvature of $\partial K$ at the point of $\partial K$ whose outer normal is $u.$
For a body $K\in \mathcal{K}^{n}$,
\[V(K)=V_1(K,K)=\frac{1}{n}\int_{\mathbb{S}^{n-1}}h_K(u)dS_K(u).\]
Of significant importance in convex geometry is the Minkowski mixed volume inequality. Minkowski's mixed volume inequality states that for $K,L\in\mathcal{K}^n,$
\[V_1(K,L)^n\geq V(K)^{n-1}V(L).\]
In the class of origin-symmetric convex bodies, equality hold if and only if $K=cL$ for some $c>0.$

In $\mathbb{R}^2$ a stronger version of Minkowski's inequality was obtained by Groemer \cite{Groemer1}. We provide his result for bodies in $\mathcal{K}_e^2:$
\begin{thmmain}[Groemer, \cite{Groemer1}]
Let $K,L\in\mathcal{K}_e^2$ and set $D(K)=2\max\limits_{\mathbb{S}^1}h_{K}$, then
\begin{align*}\frac{V(K,L)^2}{V(K)V(L)}-1&\geq \frac{V(K)}{4D^2(K)}\max_{u\in\mathbb{S}^1}\left|\frac{h_K(u)}{V(K)^{\frac{1}{2}}}-\frac{h_L(u)}{V(L)^{\frac{1}{2}}}\right|^2.
\end{align*}
\end{thmmain}
The projection body, $\Pi K\in \mathcal{K}^n_e,$ of convex body $K$ is the convex body whose support function is given by
\begin{equation}\label{e: support projection}
h_{\Pi K}(u)=\frac{1}{2}\int_{\partial K}|\langle u,v\rangle|dv,
\end{equation}
where the integration is done with respect to $(n-1)$-Hausdorff measure.
\begin{remark}\label{rem: projection dimension two}
If $L\in\mathcal{K}^2_e$, then $\Pi L= L^{\pi/2}+L^{-\pi/2}=2L^{\pi/2}$, where $h_{L^{\pi/2}}(\theta)=h_L(\theta+\pi/2)$ and $h_{L^{-\pi/2}}(\theta)=h_L(\theta-\pi/2)$, see \cite[Theorem 4.1.4]{Gardner} (Convex bodies $L^{ \pi/2}$ and $L^{-\pi/2}$ are rotations of $L$ counter-clockwise and clockwise through $90^{\circ}$ respectively.).
\end{remark}
The polar body, $K^{\ast}$, of $K\in \mathcal{K}^n_0$ is the convex body defined by
$$K^{\ast}=\{x\in\mathbb{R}^n: \langle x,y\rangle\leq 1 \mbox{~for~all~}y\in K\}.$$
A fundamental affine inequality is the Petty projection inequality which states for $K\in\mathcal{K}^n$,
\begin{equation}\label{e: projection inequality}
V(K)^{n-1} V((\Pi K)^{\ast})\leq \left(\frac{\omega_{n}}{\omega_{n-1}}\right)^n,
\end{equation}
with equality if and only if $K$ is an ellipsoid. This inequality was proved by Petty \cite{P2}, using the Busemann-Petty centroid inequality. Another proof of the Petty projection inequality using the Busemann-Petty centroid inequality is provided by Lutwak through a class reduction technique, \cite{Lutwak}. It is worth pointing out that the Petty projection inequality is a strengthened form of the classical isoperimetric inequality. Extensions of (\ref{e: projection inequality}) are given in \cite{Lutwak1,Lutwak2}.

For $x\in\operatorname{int}K$, let $K^x:=(K-x)^{\ast}.$ The Santal\'{o} point of $K$, denoted by $s$, is the unique point in int $K$ such that
$$V(K^s)\leq V(K^x)$$
for all $x\in\operatorname{int}K.$ The Blaschke-Santal\`{o} inequality \cite{BW1,Santalo} states that
\begin{equation}\label{e: blascke-santalo inequality}
V(K^s)V(K)\leq \omega_n^2,
\end{equation}
with equality if and only if $K$ is an ellipsoid. The equality condition was proved by Saint Raymond \cite{SR} in the symmetric case and Petty \cite{P1} in the general case. A proof of this inequality is also given via the affine normal flow by Andrews \cite{BA,BA2}.

The Santal\'{o} point of $K$ is characterized by the following property
$$\int_{\mathbb{S}^{n-1}}\frac{u}{h_{K-s}^{n+1}(u)}d\sigma(u)=0,$$
where $\sigma$ is the spherical Lebesgue measure on $\mathbb{S}^{n-1}.$ Thus, for an arbitrary convex body $K$, $h_{K-s}^{-(n+1)}$ satisfies the sufficient condition of Minkowski's existence theorem in $\mathbb{R}^n$. So there exists a unique convex body (up to translation), denoted by $\Lambda K$, whose surface area measure, $S_{\Lambda K}$, satisfies
\begin{equation}\label{e: def Lambda}
dS_{\Lambda K}=\frac{V(K)}{V(K^s)}h_{K-s}^{-(n+1)}d\sigma.
\end{equation}
Moreover, $\Lambda \Phi K=\Phi \Lambda  K$ (up to translation) for $\Phi \in GL(n)$, see \cite[Lemma 7.12]{Lutwak3}.
\begin{remark}\label{re: curvature image symmetric}
By the uniqueness of the solution to the even Minkowski problem, if $K\in\mathcal{K}_e^2$, then there is a unique origin-symmetric solution to (\ref{e: def Lambda}). Indeed, for each solution $\Lambda K$, we have $\Lambda K+\vec{a}\in \mathcal{K}_e^2$, for some vector $\vec{a}$ (see, for example, \cite[p. 370]{Lutwak3}). In the sequel, we always assume, without loss of generality, that if $L\in\mathcal{K}_e^2,$ then $\Lambda L\in \mathcal{K}_e^2.$  Also notice that for $K\in\mathcal{K}^2_e$ the property $\Lambda \Phi  K=\Phi \Lambda K$ (up to translation), for $\Phi  \in GL(2)$, implies that $(\Lambda K)^{\pi/2}=\Lambda K^{\pi/2}$.
\end{remark}
A useful characterization of the centroid operator is given by Lutwak \cite[Lemma 5]{Lutwak}:
\begin{equation}\label{e: centroid and projection of curvature image}
\Gamma (K-c)=\frac{2}{(n+1)V(K^c)}\Pi\Lambda K^c,
\end{equation}
where $c$ denotes the centroid of $K.$

Let $K$ be an origin-symmetric convex body, then the existence of John's ellipsoid implies $d_{\mathcal{BM}}(K,B)\leq \sqrt{n}$, \cite{J} (See also \cite{FG} for a simple proof.). In particular, this implies that for each $K\in\mathcal{K}_e^2$, there is a linear transformation $\Phi\in GL(2)$ such that $1\leq h_{\Phi K}\leq \sqrt{2}.$

Given a body $K\in\mathcal{K}^n$, the inner radius of $K$, $r_-(K)$, is the radius of the largest ball contained in $K$; the outer radius of $K$, $r_+(K)$, is the radius of the smallest ball containing $K$. For each $K\in\mathcal{K}_e^n$, the smallest and the largest balls as above will be centered at the origin of $\mathbb{R}^n$.

We conclude this section by mentioning that for $K\in\mathcal{K}^n_e$ the Santal\'{o} point and the centroid coincide with the origin of $\mathbb{R}^n$.

\section{Basic properties of the flow}\label{sec: basic prop}
Arguments of this section are standard. For completeness, we sketch their proofs. Recall that $V(K)=\displaystyle \frac{1}{2} \int_{\mathbb{S}^1}h_K\mathcal{S}_K\, d\theta$ and
$V(K^{\ast})=\displaystyle \frac{1}{2}\int_{\mathbb{S}^1}\frac{1}{h_K^2}d\theta$. The following evolution equations can be derived by direct computation.
\begin{lemma}\label{lem: ev equations} Under flow (\ref{e: asli}) one has
\begin{equation}\label{e: ev equation}
\frac{\partial}{\partial t}\mathcal{S}=-\left(h^{-2}\mathcal{S}^{-1}\right)_{\theta\theta}-h^{-2}\mathcal{S}^{-1},
\end{equation}
\begin{equation}\label{e: volume}
\frac{d}{dt}V(K_t)=-2V(K_t^{\ast}).
\end{equation}
\end{lemma}
\begin{proposition} \label{prop: ev speed}
The time dependent quantity $\displaystyle\min_{\theta\in\mathbb{S}^1}\left(h^{-2}\mathcal{S}^{-1}\right)(\theta,t)$ increases in time under (\ref{e: asli}).
\end{proposition}
\begin{proof}
Using evolution equations (\ref{e: asli}) and (\ref{e: ev equation}), we obtain
\begin{align}\label{e: laplacian}
\frac{\partial}{\partial t}\left(h^{-2}\mathcal{S}^{-1}\right)=&\left(\frac{\partial}{\partial t}h^{-2}\right)\mathcal{S}^{-1}+h^{-2}\left(\frac{\partial}{\partial t}\mathcal{S}^{-1}\right)\nonumber\\
=&\mathcal{S}^{-2}h^{-2}\left[
\left(h^{-2}\mathcal{S}^{-1}\right)_{\theta\theta}+h^{-2}\mathcal{S}^{-1}\right] \nonumber\\
&+2h^{-3}\mathcal{S}^{-1}\left(h^{-2}\mathcal{S}^{-1}\right).\nonumber
\end{align}
Standard parabolic maximum principle completes the proof.
\end{proof}
An immediate consequence is the preservation of the strict convexity.
\begin{corollary}\label{cor: Convexity is preserved}
The strict convexity of the evolving bodies is preserved as long as the flow exists.
\end{corollary}
\begin{proof}
By Proposition ~\ref{prop: ev speed} as long as the flow exists,
\[\min_{\theta\in\mathbb{S}^1}\left(\mathcal{G}/h^2\right)(\theta,t)\ge
\min_{\theta\in\mathbb{S}^1}\left(\mathcal{G}/h^2\right)(\theta,0).\]
Therefore
\[\mathcal{G}(\theta, t)\ge h^{2}(\theta, t)\min_{\theta\in\mathbb{S}^1}\left(\mathcal{G}/h^2\right)(\theta,0)> 0.\]
So the claim follows.
\end{proof}
\begin{lemma}\label{lem: upper bound for speed}
If there exists an $ r>0$ such that $h_{K_t}\ge r$ on $[0,T)$, then $\{\mathcal{G}_{K_t}\}$ is uniformly bounded from above on $[0,T)$.
\end{lemma}
\begin{proof}
We only sketch a proof that is based on Tso's trick \cite{Tso}. Define $\Omega(\theta,t):= \frac{h^{-2}\mathcal{S}^{-1}}{h-\rho}$, where $\rho=\frac{1}{2}r$. We may assume, without loss of generality, that the maximum of $\Omega$ occurs in $(0,T)$.
At the point where the maximum of $\Omega$ occurs, we have
$$\Omega_{\theta}=0, ~~ \Omega_{\theta\theta}\leq 0,$$
and
\[0\leq\frac{\partial}{\partial t}\Omega\leq\frac{1}{h-\rho}\left[h^{-2}\mathcal{S}^{-2}\left(-\frac{\rho h^{-2}\mathcal{S}^{-1}- h^{-2}}{h-\rho}\right)+\mathcal{S}^{-1}\frac{\partial}{\partial t} h^{-2}+ \frac{(h^{-4}\mathcal{S}^{-2})}{h-\rho}\right].\]
This last inequality gives
$$-\rho\mathcal{G}-2\frac{\rho}{h}+4\ge 0\Rightarrow 0<\mathcal{G}\leq \frac{4}{\rho}.$$
This gives an upper bound on $\Omega$ and thus claim follows.
\end{proof}
\begin{proposition}\label{prop: area goes to zero}
Let $T$ be the maximal time of existence of the solution to flow (\ref{e: asli}) with $K_0 \in \mathcal{K}_{e}^2$. Then $T$ is finite and $V(K_t)$ tend to zero as $t$ approaches $T$.
\end{proposition}
\begin{proof}
The first part of the claim follows from the comparison principle and the fact that the solution to (\ref{e: asli}) starting from a disk centered at the origin exists only on a finite time interval. Having established Corollary \ref{cor: Convexity is preserved}, and Lemma \ref{lem: upper bound for speed}, bounds on higher derivatives of the support function follows from Schauder theory \cite{K}, if $\lim\limits_{t\to T}V(K_t)\neq 0$. This in turn contradicts the maximality of $T$.
\end{proof}
\section{long time behavior}\label{sec: long time}
In this section we calculate the time derivative of $V(\Gamma K_t)/V(K_t)$ and we deduce the asymptotic behavior of the normalized solution, $\left(V(B)/V(K_t)\right)^{1/2}K_t$. We shall begin by rewriting the integral representation of the centroid body of $K\in\mathcal{K}^n_0$ in terms of the support function of the polar body $K^{\ast}.$ To this aim, we need first to introduce the radial function of convex body $K\in\mathcal{K}_0^n.$

The function $\rho_K:\mathbb{S}^{n-1}\to \mathbb{R}$ defined by
\[\rho_K(u):=\max\{\lambda\geq 0:\lambda u\in K\}\]
is called the radial function. This function parameterizes $\partial K$ over the unit sphere by
\[X_{K}=\rho(u)u:\mathbb{S}^{n-1}\to\mathbb{R}^n.\]
It can be shown that $\rho$ is a Lipchitz function. Moreover, $\rho_K(u)=1/h_{K^{\ast}}(u)$, see \cite[Theorem 1.7.6]{Schneider}. Thus, we may rewrite (\ref{e: support centroid}) for $K\in \mathcal{K}_0^n$ as follows
\begin{align*}
h_{\Gamma K}(u)=&\frac{1}{V(K)}\int_{K}|\langle u,x\rangle|dx=\frac{1}{(n+1)V(K)}\int_{\mathbb{S}^{n-1}}|\langle u,v\rangle|\rho_K^{n+1}(v)d\sigma(v)\\
=&\frac{1}{(n+1)V(K)}\int_{\mathbb{S}^{n-1}}|\langle u,v\rangle|h_{K^{\ast}}^{-(n+1)}(v)d\sigma(v).
\end{align*}
From this last integral representation of $h_{\Gamma K}$ it is evident that to calculate the time derivative of $V(\Gamma K_t)$, we would need first to calculate the time derivative of
$h_{K^{\ast}_t}.$
\begin{lemma}\label{lem: ev polar}
Let $K_t$ evolve by (\ref{e: asli}). Then polar bodies, $K^{\ast}_t$, evolve according to
\[\partial_t h_{K^\ast_t}(u)=h_{K^\ast_t}^4(u)\mathcal{S}_{K^\ast_t}(u).\]
\end{lemma}
Employing the evolution equation of polar bodies was first introduced by Stancu \cite{S} in the context of centro-affine normal flows.
In \cite{Ivaki3, IS} it was shown that the evolution equation of polar bodies combined with Tso's trick and Salkowski-Kaltenbach-Hug identity (see \cite[Theorem 2.8]{Hug}) provide a useful tool for obtaining regularity of solutions to a class of geometric flows. Proof of Lemma \ref{lem: ev polar} is omitted because of its similarity to the proof of \cite[Theorem 2.2]{Ivaki3}. We will now turn to the evolution equation of $h_{\Gamma K_t}.$
\begin{lemma}\label{lem: ev support of centroid body}
Let $K_t$ evolve by (\ref{e: asli}). Then centroid bodies, $\Gamma K_t$, evolve according to
\[\partial_t h_{\Gamma K_t}(u)=\frac{2V(K_t^{\ast})}{V(K_t)}h_{\Gamma K_t}(u)-\frac{2}{V(K_t)}h_{\Pi K_{t}^{\ast}}(u).\]
\end{lemma}
\begin{proof}
We will use the evolution equation of $V(K_t)$ in Lemma \ref{lem: ev equations} and the evolution equation of $h_{K_t^{\ast}}$ stated in Lemma \ref{lem: ev polar}.
Since for a small neighborhood of $t$, $\partial_t h_{K_t^{\ast}} $ are bounded and $h_{K_t^{\ast}}$ are bounded from above and from below by positive numbers, letting $\partial_t$ to commute with $\int_{\mathbb{S}^1}$ is justified.
\begin{align*}
&\partial_th_{\Gamma K_t}(u)\\
&=\partial_t\left(\frac{1}{3V(K_t)}\int_{\mathbb{S}^{1}}|\langle u,v\rangle|h_{K^{\ast}_t}^{-3}(v)d\sigma(v)\right)\\
&=\frac{-\frac{d}{dt}V(K_t)}{3V^2(K_t)}\int_{\mathbb{S}^{1}}|\langle u,v\rangle|h_{K^{\ast}_t}^{-3}(v)d\sigma(v)+
\frac{1}{3V(K_t)}\int_{\mathbb{S}^{1}}|\langle u,v\rangle|\partial_t\left(h_{K^{\ast}_t}^{-3}(v)\right)d\sigma(v)\\
&=\frac{2V(K_t^{\ast})}{3V^2(K_t)}\int_{\mathbb{S}^{1}}|\langle u,v\rangle|h_{K^{\ast}_t}^{-3}(v)d\sigma(v)-
\frac{1}{V(K_t)}\int_{\mathbb{S}^{1}}|\langle u,v\rangle|\mathcal{S}_{K_{t}^{\ast}}(v)d\sigma(v)\\
&=\frac{2V(K_t^{\ast})}{V(K_t)}h_{\Gamma K_t}(u)-\frac{2}{V(K_t)}h_{\Pi K_{t}^{\ast}}(u).
\end{align*}
On the last line, we used the fact that we may rewrite (\ref{e: support projection}) for $K\in\mathcal{C}^2_{+}$ as follows
\begin{equation*}
h_{\Pi K}(u)=\frac{1}{2}\int_{\partial K}|\langle u,v\rangle|dv=\frac{1}{2}\int_{\mathbb{S}^1}|\langle u,v\rangle|\mathcal{S}(v)d\sigma(v).
\end{equation*}
\end{proof}
\begin{lemma}\label{lem: ev area of centroid body}
As $K_t$ evolve by (\ref{e: asli}), centroid bodies, $\Gamma K_t$, evolve according to
\[\frac{d}{dt} V(\Gamma K_t)=\frac{4V(K_t^{\ast})}{V(K_t)}V(\Gamma K_t)-\frac{4}{V(K_t)}V(\Gamma K_t,\Pi K_{t}^{\ast}).\]
\end{lemma}
\begin{proof}
Recall that if $K\in\mathcal{K}^n$ is $\mathcal{C}_+^2$, then $h_{K}$ is two times continuously differentiable (see \cite[p. 106]{Schneider}).
As $\Gamma K_t\in\mathcal{C}_+^2$ we can write
\[V(\Gamma K_t)=\frac{1}{2}\int_{\mathbb{S}^1}h_{\Gamma K_t}\left((h_{\Gamma K_t})_{\theta\theta}+h_{\Gamma K_t}\right)d\theta.\]
Furthermore, convex bodies $\Pi K_t^{\ast}=2(K_t^{\ast})^{\pi/2}$ are also $\mathcal{C}_+^2$.
The claim now follows from Lemma \ref{lem: ev support of centroid body}, integration by parts, and the definition of the mixed volume.
\end{proof}
\begin{corollary}\label{cor: monotonicity of petty-centroid}
$\frac{d}{dt}\frac{V(\Gamma K_t)}{V(K_t)}\leq 0$ with equality if and only if $K_t$ is an origin-centered ellipse. In particular, if $K\in\mathcal{K}^2_e$ is smooth and minimizes $\frac{V(\Gamma K)}{V(K)}$, then $K$ must be an origin-centered ellipse.
\end{corollary}
\begin{proof} Lemmas \ref{lem: ev equations} and \ref{lem: ev area of centroid body} yield
\begin{align*}
\frac{d}{dt}\frac{V(\Gamma K_t)}{V(K_t)}=&\frac{6V(K_t^{\ast})}{V^2(K_t)}V(\Gamma K_t)-\frac{4}{V^2(K_t)}V(\Gamma K_t,\Pi K_{t}^{\ast}).
\end{align*}
Replacing $\Gamma K_t$ on the right-hand side by its equivalent expression from (\ref{e: centroid and projection of curvature image}) will lead us to
\begin{align}\label{e: last step time derv centroid}
\frac{d}{dt}\frac{V(\Gamma K_t)}{V(K_t)}=&\frac{8}{3V^2(K_t)V(K_t^{\ast})}\left(V(\Pi\Lambda K_t^{\ast})-V(\Pi\Lambda K_t^{\ast},\Pi K_{t}^{\ast})\right).
\end{align}
By Remarks \ref{rem: projection dimension two} and \ref{re: curvature image symmetric} we may rewrite
(\ref{e: last step time derv centroid}) as follows
\begin{align}\label{ie: asym behavior1}
\frac{d}{dt}\frac{V(\Gamma K_t)}{V(K_t)}=&\frac{32}{3V^2(K_t)V(K_t^{\ast})}\left(V((\Lambda K_t^{\ast})^{\pi/2})-V(\Lambda (K_t^{\ast})^{\pi/2},(K_{t}^{\ast})^{\pi/2})\right)\nonumber\\
=&\frac{32}{3V^2(K_t)V(K_t^{\ast})}\left(V((\Lambda K_t^{\ast})^{\pi/2})-V((K_{t}^{\ast})^{\pi/2})\right)\nonumber\\
=&\frac{32}{3V^2(K_t)V(K_t^{\ast})}\left(V(\Lambda K_t^{\ast})-V(K_t^{\ast})\right).
\end{align}
Here we used the easily established identity $V(\Lambda L,L)=V(L)$ for $L\in \mathcal{K}^2$ (This identity follows from the definition of $\Lambda L$ and the definition of the mixed volume, see \cite[Lemma 3]{Lutwak}.). Notice that by the Minkowski inequality $V^2(L)=V^2(\Lambda L,L)\geq V(L)V(\Lambda L).$ Therefore, $V(L)\geq V(\Lambda L)$ for all $L\in \mathcal{K}^2,$ with equality if and only if $\Lambda L$ is a translation of $L$. For $L\in\mathcal{K}^2_e$, equality is achieved only if $\Lambda L=L$, as $\Lambda L\in\mathcal{K}^2_e$. Moreover, \cite[Lemma 8.1]{P1} states that $\Lambda L=L$ if and only if $L$ is an origin-centered ellipse.
\end{proof}
\begin{proposition}\label{prop: 1 step to asym shape} As $K_t$ evolve by (\ref{e: asli}) the following limit holds:
\[\limsup_{t\to T}\frac{1}{V(K_t)V^2(K_t^{\ast})}\left(V(\Lambda K_t^{\ast})-V(K_t^{\ast})\right)=0.\]
\end{proposition}
\begin{proof}
Suppose on the contrary, that there exist $\varepsilon,~\delta>0$ such that for all $t\in (T-\delta,T)$ we have
\[\frac{\left(V(\Lambda K_t^{\ast})-V(K_t^{\ast})\right)}{V(K_t)V^2(K_t^{\ast})}\leq -\varepsilon.\]
Therefore, by Lemma \ref{lem: ev equations} and inequality (\ref{ie: asym behavior1}), we get
\[\frac{d}{dt}\frac{V(\Gamma K_t)}{V(K_t)}\leq \frac{16}{3}\varepsilon\frac{d}{dt}\ln(V(K_t)).\]
Since $\lim\limits_{t\to T}V(K_t)=0$ by Lemma \ref{prop: area goes to zero}, we deduce that
\[\lim_{t\to T} \frac{V(\Gamma K_t)}{V(K_t)}=-\infty.\]
However $V(\Gamma K_t)/V(K_t)$ is manifestly positive.
\end{proof}
\subsection{Convergence of a subsequence of the normalized solution in the Hausdorff metric}\label{sec: Convergence in the Hausdorff metric}
We begin this section by recalling  the weak continuity of surface area measures. If $\{K_i\}_i$ is a sequence in $\mathcal{K}^n_0$, then
\[\lim_{i\to\infty}K_i=K\in\mathcal{K}_0^2\Rightarrow \lim_{i\to\infty}S_{K_i}=S_K,~~\mbox{weakly}.\]
Weak convergence means that for every continuous function $f$ on $\mathbb{S}^{n-1}$ we have
\[\lim_{i\to\infty}\int_{\mathbb{S}^{n-1}}fdS_{K_i}=\int_{\mathbb{S}^{n-1}}fdS_{K}.\]
\begin{lemma}\label{lem: cont curvature image}
If $\{K_i\}$ is a sequence of bodies in $\mathcal{K}^2_e$ converging to a body $K\in\mathcal{K}^2_e$, then
\[\lim_{i\to\infty}\Lambda K_{i}=\Lambda K.\]
\end{lemma}
\begin{proof}
On the one hand, by the definition of $\Lambda K$, for every continuous function $f$ on $\mathbb{S}^{1}$ we have
\begin{align}\label{e: cont curvature image}
\lim_{i\to\infty}\int_{\mathbb{S}^{1}}fdS_{\Lambda K_i}=&\lim_{i\to\infty}\left(\frac{V(K_i)}{V(K^{\ast}_i)}\int_{\mathbb{S}^{1}}fh_{K_i}^{-3}d\sigma\right)\\
=&\frac{V(K)}{V(K^{\ast})}\int_{\mathbb{S}^{1}}fh_{K}^{-3}d\sigma=\int_{\mathbb{S}^{1}}fdS_{\Lambda K}.\nonumber
\end{align}
Notice that to go from the first line to the second line, we have used the bounded convergence theorem: Indeed, $\lim\limits_{i\to\infty}K_i=K$ implies that $\{h_{K_i}\}$ converges uniformly on $\mathbb{S}^1$ to $h_K$. Moreover, by the assumption $K\in\mathcal{K}^2_e$, we have $0<m<h_{K}<M<\infty$, for some constants. Therefore, $\{1/h_{K_i}\}$ is uniformly bounded from above.

On the other hand, $\{\Lambda K_i\}$ is uniformly bounded: Recall from identity (\ref{e: centroid and projection of curvature image}) that ${\Pi\Lambda K_i}=\frac{3}{2}V(K_i){\Gamma K_i^{\ast}}$ and it is clear that $\{\Gamma K_i^{\ast}\}$ is uniformly bounded. Therefore $\{\Pi\Lambda K_i=2(\Lambda K_i)^{\pi/2}\}$ is uniformly bounded. Thus, every a priori chosen subsequence of $\{\Lambda K_i\}$ by Blaschke's selection theorem has a subsequence, denoted by $\{i_k\}_{k}$, such that $\{\Lambda K_{i_k}\}$ converges to a body $L\in\mathcal{K}^2_e.$ Thus $\lim\limits_{k\to\infty}S_{\Lambda K_{i_k}}=S_{L}$, weakly.
Putting these two observations altogether, we infer
\begin{align}\label{e: even uniq}
\int_{\mathbb{S}^{1}}fdS_{L}=\lim_{k\to\infty}\int_{\mathbb{S}^{1}}fdS_{\Lambda K_{i_k}}=\int_{\mathbb{S}^{1}}fdS_{\Lambda K},
\end{align}
for every continuous function $f$ on $\mathbb{S}^{1}$. In particular, (\ref{e: even uniq}) implies that
\[V(L,Q)=V(\Lambda K,Q), \mbox{~for~every~} Q\in\mathcal{K}^2.\]
Therefore, $V^2(\Lambda K,L)=V^2(L)=V^2(\Lambda K)=V(L)V(\Lambda K)$. So by the equality case in Minkowski's mixed volume inequality $L=\Lambda K.$ In particular, the limit $L$ is independent of the subsequence we have chosen a priori. The proof is now complete.
\end{proof}
We now have all necessary ingredients to prove the weak convergence (equivalently, convergence in the Hausdorff metric) of a subsequence of the normalized solution.
\begin{proof}
By Proposition \ref{prop: 1 step to asym shape}, there exists a sequence of times, $\{t_k\}_{k}$, such that $t_k\to T$ and
\begin{align}\label{e: last}
\lim_{t_k\to T}\left(\frac{1}{V(K_{t_k})V(K_{t_k}^{\ast})}\left(\frac{V(\Lambda K_{t_k}^{\ast})}{V(K_{t_k}^{\ast})}-1\right)=:\chi(t_k)\right)=0.
\end{align}
Observe that $\chi(t_k)$, $V(\Lambda K_{t_k}^{\ast})/V(K_{t_k}^{\ast})$, and $V(K_{t_k})V(K_{t_k}^{\ast})$  are all invariant under $GL(2)$.
Define $\bar{K}_t:=\left(V(B)/V(K_t^{\ast})\right)^{1/2}K_t^{\ast}$ and for each $t$ let $\Psi_t\in SL(2)$ be a linear transformation that minimizes the length of $\partial \bar{K}_t.$
By John's ellipsoid theorem the sequence $\{\Psi_{t_k}\bar{K}_{t_k}\}$ is uniformly bounded. Thus, by Blaschke's selection theorem and by passing to a subsequence of times denoted again by $\{t_k\}_{k}$ we deduce that there exists a body $\bar{K}_{\infty}\in\mathcal{K}^2_e$ such that $\lim\limits_{t_k\to T}\Psi_{t_k}\bar{K}_{t_k}=\bar{K}_{\infty}.$ In light of Lemma \ref{lem: cont curvature image}, we have $\lim\limits_{t_k\to T}\Lambda \Psi_{t_k}\bar{K}_{t_k}=\Lambda\bar{K}_{\infty}.$ So from (\ref{e: last}) we get
\[\lim_{t_k\to T}\frac{V(\Lambda \Psi_{t_k}\bar{K}_{t_k})}{V(\Psi_{t_k}\bar{K}_{t_k})}=\frac{V(\Lambda \bar{K}_{\infty})}{V(\bar{K}_{\infty})}=1.\]
Recall that $V(\Lambda \bar{K}_{\infty})=V(\bar{K}_{\infty})$ implies that $\Lambda \bar{K}_{\infty}=\bar{K}_{\infty}$. Now Lemma 8.1 of Petty \cite{P1} yields that $\bar{K}_{\infty}$ must be an origin-centered ellipse. Since the length of $\Psi_{t_k}\bar{K}_{t_k}$ is minimized and $V(\Psi_{t_k}\bar{K}_{t_k})=\pi$, $\bar{K}_{\infty}$ must be the unit disk. Consequently
$\lim\limits_{t_k\to T}\Psi_{t_k}\bar{K}_{t_k}=B$ and
\begin{equation}\label{e: aa}
\lim\limits_{t_k\to T}(\Psi_{t_k}\bar{K}_{t_k})^{\ast}=B, \lim\limits_{t_k\to T}V((\Psi_{t_k}\bar{K}_{t_k})^{\ast})=\pi.
\end{equation}
Let $\Phi_{t_k}$ be the inverse of the transpose of $\Psi_{t_k}$, for each $t_k$. So we get
 \[(\Psi_{t_k}\bar{K}_{t_k})^{\ast}=\left(\frac{V(B)}{V(K_{t_k}^{\ast})}\right)^{-\frac{1}{2}}\Phi_{t_k}K_{t_k}=\left(\frac{V^2(B)}{V(K_{t_k})V(K_{t_k}^{\ast})}\right)^{-\frac{1}{2}}
 \left(\frac{V(B)}{V(K_{t_k})}\right)^{\frac{1}{2}}\Phi_{t_k}K_{t_k}.\]
 Combining (\ref{e: aa}) with  $V\left( \left(\frac{V(B)}{V(K_{t_k})}\right)^{\frac{1}{2}}\Phi_{t_k}K_{t_k}\right)=\pi$ implies
 \[\lim\limits_{t_k\to T}\left(\frac{V^2(B)}{V(K_{t_k})V(K_{t_k}^{\ast})}\right)^{-\frac{1}{2}}=1,\] and in turn
 \[\lim_{t_k\to T}\left(\frac{V(B)}{V(K_{t_k})}\right)^{\frac{1}{2}}\Phi_{t_k}K_{t_k}=B.\]
\end{proof}
\section{Stability of the Planar Busemann-Petty centroid inequality}\label{sec: stability}
We will state several lemmas from \cite{BA4,Ivaki,Ivaki2,Ivaki4} to prepare the proof of Theorem B.
The first lemma is rewriting identity (\ref{ie: asym behavior1}).
\begin{lemma}\label{lem: time dev of the ratio}
Along flow (\ref{e: asli}) we have
\begin{align}\label{ie: asym behavior}
\frac{d}{dt}\frac{V(\Gamma K_t)}{V(K_t)}=&\frac{32V(\Lambda K_t^{\ast})}{3V^2(K_t)V(K_t^{\ast})}\left(1-\frac{V^2(\Lambda K_t^{\ast},K_{t}^{\ast})}{V(\Lambda K_t^{\ast})V(K_t^{\ast})}\right).
\end{align}
\end{lemma}
\begin{lemma}\label{lem: app} Under the evolution equation (\ref{e: asli}) we have
\[h_{K_t}(u)\leq h_{K_{0}}(u)\leq h_{K_t}(u)\left(1+2t\left(\frac{\mathcal{G}}{h^3}\right)(u,t)\right).\]
In particular,
\[d_{\mathcal{BM}}(K_0,K_t)\leq \left(1+2t\max_{u\in\mathbb{S}^1}\left(\frac{\mathcal{G}}{h^3}\right)(u,t)\right).\]
\end{lemma}
To prove Lemma \ref{lem: app}, one may first obtain a Harnack estimate using the method of \cite{BA1} from which the right-hand side follows. The left-hand side holds trivially as the flow is a shrinking one. The details are given by the author for the $p$-centro affine normal flows in \cite{Ivaki4,Ivaki5} (See also Lemmas \ref{lem: Harnack est} and \ref{lem: app1}.).
\begin{lemma}\label{lem: upper G}
For any smooth, strictly convex solution $\{K_t\}$ of evolution equation (\ref{e: asli}) with $0<R_{-}\leq h_{K_t}\leq R_{+}< \infty$, for $t\in[0,\delta]$, and some positive numbers $R_{\pm}$, we have
\[\mathcal{G}_{K_t}\leq C_0+C_1t^{-1/2},\]
where $C_0,C_1$ are constants depending on $R_{-}$ and $R_{+}.$
\end{lemma}
\begin{proof} We apply Tso's trick \cite{Tso}. Consider the function
\[\Omega=\frac{h^{-2}\mathcal{S}^{-1}}{h-R_{-}/2}.\]
Using the maximum principle, we will show that $\Omega$ is bounded from above by a function of $R_{-}, R_{+},$ and time.
At the point where the maximum of $\Omega$ occurs, we have
\[0=\Omega_{\theta}= \left(\frac{h^{-2}\mathcal{S}^{-1}}{h-R_{-}/2}\right)_{\theta}\ \ \ {\hbox{and}}\ \ \  \Omega_{\theta\theta}\leq 0.\]
Hence, we obtain
\[\frac{(h^{-2}\mathcal{S}^{-1})_{\theta}}{h-R_-/2}=\frac{(h^{-2}\mathcal{S}^{-1}) h_{\theta}}{(h-R_-/2)^2}\]
 and consequently
\begin{equation}\label{e: tso}
\left(h^{-2}\mathcal{S}^{-1}\right)_{\theta\theta}+\left(h^{-2}\mathcal{S}^{-1}\right)\leq
\frac{h^{-2}-(R_{-}/2)h^{-2}\mathcal{S}^{-1}}{h-R_{-}/2}.
\end{equation}
We calculate
\begin{align*}
\partial_t\Omega&=\frac{h^{-2}\mathcal{S}^{-2}}{h-R_{-}/2}
\left[\left(h^{-2}\mathcal{S}^{-1}\right)_{\theta\theta}+\left(h^{-2}\mathcal{S}^{-1}\right)\right]\\
&+\frac{\mathcal{S}^{-1}}{h-R_{-}/2}\partial_t h^{-2}+\frac{h^{-4}\mathcal{S}^{-2}}{(h-R_{-}/2)^2}.
\end{align*}
 Notice that
\begin{equation}\label{e: one to last}
\frac{\mathcal{S}^{-1}}{h-R_{-}/2}\partial_t h^{-2}=2\Omega^2-2\frac{R_{-}}{2}\frac{h^{-5}\mathcal{S}^{-2}}{(h-R_-/2)^2}\leq 2\Omega^2.
\end{equation}
Thus, using inequalities (\ref{e: tso}) and (\ref{e: one to last}), at the point where the maximum of $\Omega$ is reached, we have
\begin{equation}\label{e: last step tso}
\partial_t\Omega\leq\Omega^2\left(4-\frac{ R_{-}}{2}\mathcal{S}^{-1}\right).
\end{equation}
We can control $\mathcal{G}$ from below by a positive power of $\Omega:$
\begin{equation*}
\mathcal{S}^{-1}= \left(\frac{h-R_{-}/2}{h^{-2}\mathcal{S}^{-1}}\right)^{-1}\left(\frac{h^{-2}}{h-R_{-}/2}\right)^{-1}
\geq \Omega \left(\frac{R_{-}^{-2}}{R_{-}-R_{-}/2}\right)^{-1}.
\end{equation*}
Therefore, we can rewrite the inequality (\ref{e: last step tso}) as follows
\begin{align*}
\partial_t\Omega&\leq -\Omega^2\left(\frac{R_-^4}{4}\Omega-4\right).
\end{align*}
Hence
\[\Omega\leq C(R_{-},R_+)t^{-\frac{1}{2}}+C'(R_{-},R_+)\]
for some positive constants $C,C'$ depending on $R_-,R_+$. The corresponding claim for $\mathcal{G}$ follows.
\end{proof}
\begin{corollary}\label{cor: displacement bound}
For any solution $\{K_t\}\subset\mathcal{F}_e^2$ to (\ref{e: asli}) with $0<R_{-}\leq h_{K_t}\leq R_{+}< \infty$, for $t\in[0,\delta]$, and for some positive numbers $R_{\pm}$, we have
\[h_{K_t}(u)\leq h_{K_{0}}(u)\leq h_{K_t}(u)\left(1+\frac{2C_0}{R_-^3}t+\frac{2C_1}{R_-^3}t^{1/2}\right),\]
where $C_0, C_1$ are constants depending on $R_{-}$ and $R_{+}.$ In particular, we have
\[d_{\mathcal{BM}}(K_0,K_t)\leq \left(1+\frac{2C_0}{R_-^3}t+\frac{2C_1}{R_-^3}t^{1/2}\right).\]
\end{corollary}
\begin{proof}
The claim immediately follows from Lemmas \ref{lem: app} and \ref{lem: upper G}.
\end{proof}
\begin{lemma}\label{lem: lower and upper bound for affine associated}
For any solution $\{K_t\}\subset\mathcal{F}_e^2$ to (\ref{e: asli}) with $0<R_{-}\leq h_{K_t}\leq R_{+}< \infty$, for $t\in[0,\delta]$, and for some positive numbers $R_{\pm}$, we have
\[C_2\leq h_{\Lambda K^{\ast}_t}~\& ~h_{K^{\ast}_t}\leq C_3,\]
where $C_2$ and $C_3$ are constants depending on $R_{-}$ and $R_{+}.$
\end{lemma}
\begin{proof}
The claim for $K^{\ast}_t$ is trivial. To prove the claim for $\Lambda K^{\ast}_t$, first notice that $(\Lambda K^{\ast}_t)^{\pi/2}=\frac{1}{2}\Pi\Lambda K^{\ast}_t=
\frac{3}{4}V(K^{\ast}_t)\Gamma K_t.$ Moreover, recall that $V(K^{\ast})=\frac{1}{2}\int_{\mathbb{S}^1}h_K^{-2}d\theta.$ Thus, uniform lower and upper bounds on the support functions of $\Lambda K^{\ast}_t$ follow from the definition of $\Gamma( \cdot).$
\end{proof}
\begin{lemma}[\cite{Ivaki}]\label{lem: ellipsoid app} Suppose that $K\in\mathcal{F}_{e}^2$. If $m\le\frac{h}{\mathcal{G}^{1/3}}\le M$, for some positive numbers $m$ and $M$, then there exist two origin-symmetric ellipses $E_{in}$ and $E_{out}$ such that $E_{in}\subseteq K\subseteq E_{out}$ and
\[\left(\frac{V(E_{in})}{\pi}\right)^{2/3}=m,~~ \left(\frac{V(E_{out})}{\pi}\right)^{2/3}=M.\]
Moreover, we have
\[d_{\mathcal{BM}}(K,B)\leq \left(\frac{M}{m}\right)^{\frac{3}{2}}.\]
\end{lemma}
\begin{proof}
A proof of the first part of the claim is given in \cite{Ivaki}. To prove the second part of the claim, we may first apply a special linear transformation $\Phi\in SL(2)$ such that $\Phi E_{out}$ is a disk. Then it is easy to see that $\Phi E_{out}\subseteq \frac{V(E_{out})}{V(E_{in})}\Phi E_{in}.$ Therefore
\[\Phi E_{in} \subseteq  \Phi K\subseteq \frac{V(E_{out})}{V(E_{in})} \Phi E_{out},\]
and
\[d_{\mathcal{BM}}(K,B)\leq \frac{V(E_{out})}{V(E_{in})}.\]
\end{proof}
A simple consequence of Lemma \ref{lem: ellipsoid app} is contained in the following corollary.
\begin{corollary}[\cite{BA4,Ivaki2}]\label{cor: minmax of affine}
Let $K\in\mathcal{F}_{e}^2$ be of area $\pi$. Then
$$\min_{\mathbb{S}^1} \frac{h}{\mathcal{G}^{1/3}}\leq 1\leq\max_{\mathbb{S}^1}\frac{h}{\mathcal{G}^{1/3}}.$$
\end{corollary}
For another proof of Corollary (\ref{cor: minmax of affine}), see Andrews \cite[Lemma 10]{BA4} where he does not assume that $K$ is origin-symmetric.
\subsection{Proof of Theorem B}
In this section, $C_4,C_5,\cdots$ are absolute positive constants. Moreover, we will repeatedly use Lemma \ref{lem: lower and upper bound for affine associated} without further mention.

Let $K\in\mathcal{F}_e^2$ be a body such that
\[\frac{V(\Gamma K)}{V(K)}\leq \left(\frac{4}{3\pi}\right)^2(1+\varepsilon),~~0<\varepsilon<\varepsilon_0.\]
The value of $\varepsilon_0$ will be determined later.

Let $\Phi\in GL(2)$ such that $1\leq h_{\Phi K}\leq \sqrt{2}$ and let $\{K_t\}$ be the solution to flow (\ref{e: asli}) with $K_0=\Phi K.$
It follows from the comparison principle that there exists $\delta>0$, independent of $K_0$, such that $1/2\leq h_{K_t}\leq \sqrt{2}$, for all $t\in[0,\delta].$
From Corollary \ref{lem: time dev of the ratio} we have
\[\int_{0}^{f(\varepsilon)}\frac{d}{dt}\frac{V(\Gamma K_t)}{V(K_t)}dt=\int_{0}^{f(\varepsilon)}\frac{32V(\Lambda K_t^{\ast})}{3V^2(K_t)V(K_t^{\ast})}\left(1-\frac{V^2(\Lambda K_t^{\ast},K_{t}^{\ast})}{V(\Lambda K_t^{\ast})V(K_t^{\ast})}\right)dt,\]
where $f:[0,1]\to [0,\delta]$ is an increasing, continuous function such that $\lim\limits_{\varepsilon \to 0}\varepsilon/f(\varepsilon)=0$, and $\lim\limits_{\varepsilon \to 0}f(\varepsilon)=0.$
Let $s\in [0,f(\varepsilon)]$ be the time that the integrand of the right hand side is maximized. Therefore, we get
\begin{align}\label{ie: support diff0}
\int_{0}^{f(\varepsilon)}\frac{d}{dt}\frac{V(\Gamma K_t)}{V(K_t)}dt\leq f(\varepsilon)\left[\frac{32V(\Lambda K_s^{\ast})}{3V^2(K_s)V(K_s^{\ast})}\left(1-\frac{V^2(\Lambda K_s^{\ast},K_{s}^{\ast})}{V(\Lambda K_s^{\ast})V(K_s^{\ast})}\right)\right].
\end{align}
By the Busemman-Petty centroid inequality and the stability of Minkowski's mixed volume inequality, we get
\begin{align}\label{ie: support diff}
-\left(\frac{4}{3\pi}\right)^2\varepsilon &\leq -\frac{8}{3}\frac{f(\varepsilon)V(\Lambda K_s^{\ast})}{V^2(K_s)D^2(K_s^{\ast})}\max_{u\in\mathbb{S}^1}\left|\frac{h_{\Lambda K_s^{\ast}}(u)}{V(\Lambda K_s^{\ast})^{1/2}}-
\frac{h_{K_s^{\ast}}(u)}{V(K_s^{\ast})^{1/2}}\right|^2\\
&\leq -C_4 f(\varepsilon)\max_{u\in\mathbb{S}^1}\left|\frac{h_{\Lambda K_s^{\ast}}(u)}{h_{K_s^{\ast}}(u)}-\left(\frac{V(\Lambda K_s^{\ast})}{V(K_s^{\ast})}\right)^{1/2}\right|^2.\nonumber
\end{align}
By the definition of the operator $\Lambda$:
\[\frac{1}{\sqrt{2}}\leq\left(\mathcal{G}_{\Lambda K_s^{\ast}}\frac{V(K_s^{\ast})}{V(K_s)}\right)^{\frac{1}{3}}=h_{K_s^{\ast}}\leq 2.\]
Combining this with inequality (\ref{ie: support diff}) we conclude that
\[C_5\left(\frac{\varepsilon}{f(\varepsilon)}\right)^{1/2}\geq \max_{\mathbb{S}^1}\frac{h_{\Lambda K_s^{\ast}}}{\mathcal{G}^{\frac{1}{3}}_{\Lambda K_s^{\ast}}}-\min_{\mathbb{S}^1}\frac{h_{\Lambda K_s^{\ast}}}{\mathcal{G}^{\frac{1}{3}}_{\Lambda K_s^{\ast}}}.\]
As $V\left(\sqrt{\frac{\pi}{V(\Lambda K_s^{\ast})}}\Lambda K_s^{\ast}\right)=\pi$, by Corollary \ref{cor: minmax of affine}, we get
\[C_5\left(\frac{\varepsilon}{f(\varepsilon)}\right)^{1/2}\left(\frac{\pi}{V(\Lambda K_s^{\ast})}\right)^{2/3}+1\geq \left(\frac{\pi}{V(\Lambda K_s^{\ast})}\right)^{2/3}\max_{\mathbb{S}^1}\frac{h_{\Lambda K_s^{\ast}}}{\mathcal{G}^{\frac{1}{3}}_{\Lambda K_s^{\ast}}}\]
\[1-C_5\left(\frac{\varepsilon}{f(\varepsilon)}\right)^{1/2}\left(\frac{\pi}{V(\Lambda K_s^{\ast})}\right)^{2/3}\leq \left(\frac{\pi}{V(\Lambda K_s^{\ast})}\right)^{2/3}\min_{\mathbb{S}^1}\frac{h_{\Lambda K_s^{\ast}}}{\mathcal{G}^{\frac{1}{3}}_{\Lambda K_s^{\ast}}}.\]
We take $\varepsilon_0$ small enough such that
\[1-C_5\left(\frac{\varepsilon}{f(\varepsilon)}\right)^{1/2}\left(\frac{\pi}{V(\Lambda K_s^{\ast})}\right)^{2/3}\geq 1-C_{6}\left(\frac{\varepsilon}{f(\varepsilon)}\right)^{1/2}>0.\]
So we have proved: If $\varepsilon_0>0$ is small enough, then
 \[\max_{\mathbb{S}^1}\frac{h_{\Lambda K_s^{\ast}}}{\mathcal{G}^{\frac{1}{3}}_{\Lambda K_s^{\ast}}}\leq \left(1+C_{6}\left(\frac{\varepsilon}{f(\varepsilon)}\right)^{1/2}\right)\left(\frac{\pi}{V(\Lambda K_s^{\ast})}\right)^{-2/3}\]
and
\[\min_{\mathbb{S}^1}\frac{h_{\Lambda K_s^{\ast}}}{\mathcal{G}^{\frac{1}{3}}_{\Lambda K_s^{\ast}}}\geq \left(1-C_{6}\left(\frac{\varepsilon}{f(\varepsilon)}\right)^{1/2}\right)\left(\frac{\pi}{V(\Lambda K_s^{\ast})}\right)^{-2/3}.\]
From these last inequalities and Lemma \ref{lem: ellipsoid app} we deduce that
\begin{equation}\label{e: bm1}
d_{\mathcal{BM}}(\Lambda K_s^{\ast},B)\leq \left(\frac{1+C_{6}\left(\frac{\varepsilon}{f(\varepsilon)}\right)^{1/2}}{1-C_{6}\left(\frac{\varepsilon}{f(\varepsilon)}\right)^{1/2}}\right)^{3/2}.
\end{equation}
On the other hand, from (\ref{ie: support diff0}) and the Busemann-Petty centroid inequality we get
\begin{align*}
-\left(\frac{4}{3\pi}\right)^2\varepsilon &\leq f(\varepsilon)\left[\frac{32V(\Lambda K_s^{\ast})}{3V^2(K_s)V(K_s^{\ast})}\left(1-\frac{V^2(\Lambda K_s^{\ast},K_{s}^{\ast})}{V(\Lambda K_s^{\ast})V(K_s^{\ast})}\right)\right]\nonumber\\
&= f(\varepsilon)\left[\frac{32V(\Lambda K_s^{\ast})}{3V^2(K_s)V(K_s^{\ast})}\left(1-\frac{V(K_s^{\ast})}{V(\Lambda K_s^{\ast})}\right)\right].
\end{align*}
Thus
\begin{equation}\label{ie: volume ratio stability}
1\leq \frac{V(K_s^{\ast})}{V(\Lambda K_s^{\ast})}\leq 1+C_7\frac{\varepsilon}{f(\varepsilon)}.
\end{equation}
By the first line of inequality (\ref{ie: support diff}):
\begin{align*}
1-C_{8}\left(\frac{\varepsilon}{f(\varepsilon)}\right)^{1/2}&\leq
\left(\frac{V(K_s^{\ast})}{V(\Lambda K_s^{\ast})}\right)^{1/2}-C_{8}\left(\frac{\varepsilon}{f(\varepsilon)}\right)^{1/2}\\
&\leq \frac{h_{K_s^{\ast}}}{h_{\Lambda K_s^{\ast}}}\leq \left(\frac{V(K_s^{\ast})}{V(\Lambda K_s^{\ast})}\right)^{1/2}+C_{8}\left(\frac{\varepsilon}{f(\varepsilon)}\right)^{1/2}.
\end{align*}
We take $\varepsilon_0$ small enough such that $1-C_{8}\left(\frac{\varepsilon}{f(\varepsilon)}\right)^{1/2}>0.$ So by inequalities (\ref{ie: volume ratio stability}) we obtain
\begin{equation}\label{e: bm2}
d_{\mathcal{BM}}(K_s^{\ast},\Lambda K_s^{\ast})\leq \left(\frac{\left(1+C_7\frac{\varepsilon}{f(\varepsilon)}\right)^{1/2}+C_{8}\left(\frac{\varepsilon}{f(\varepsilon)}\right)^{1/2}}
{1-C_{8}\left(\frac{\varepsilon}{f(\varepsilon)}\right)^{1/2}}\right).
\end{equation}
Combining (\ref{e: bm1}) and (\ref{e: bm2}) we get
\[d_{\mathcal{BM}}(K_s^{\ast},B)\leq g(\varepsilon),\]
where
\[g(\varepsilon):=\left(\frac{1+C_{6}\left(\frac{\varepsilon}{f(\varepsilon)}\right)^{1/2}}{1-C_{6}\left(\frac{\varepsilon}{f(\varepsilon)}\right)^{1/2}}\right)^{3/2}
\left(\frac{\left(1+C_7\frac{\varepsilon}{f(\varepsilon)}\right)^{1/2}+C_{8}\left(\frac{\varepsilon}{f(\varepsilon)}\right)^{1/2}}
{1-C_{8}\left(\frac{\varepsilon}{f(\varepsilon)}\right)^{1/2}}\right).\]
This in turn implies that
\begin{equation}\label{e: bm3}
d_{\mathcal{BM}}(K_s,B)\leq g(\varepsilon).
\end{equation}
By Corollary \ref{cor: displacement bound}, we have
\begin{equation}\label{e: bm4}
d_{\mathcal{BM}}(K_0,K_s)\leq 1+C_{9}f(\varepsilon)^{\frac{1}{2}}+C_{10}f(\varepsilon).
\end{equation}
Consequently, putting (\ref{e: bm3}) and (\ref{e: bm4}) together, we obtain
\[d_{\mathcal{BM}}(K,B)=d_{\mathcal{BM}}(K_0,B)\leq \left(1+C_{9}f(\varepsilon)^{\frac{1}{2}}+C_{10}f(\varepsilon)\right)g(\varepsilon).\]
In particular, with the choice $f(\varepsilon)=\varepsilon^{1/2}$, we get
\[d_{\mathcal{BM}}(K,B)\leq 1+\gamma\varepsilon^{1/4},\]
for $\varepsilon_0$ small enough and $\gamma>0$. Therefore, we have proved the claim for bodies in $\mathcal{F}_e^2.$
An approximation argument will then upgrade the result for bodies in $\mathcal{F}_e^2$ to bodies in the larger class $\mathcal{K}_e^2.$ To get the more general result, we will first need to recall Theorem 1.4 of B\"{o}r\"{o}czky in \cite{B} and a theorem of Campi and Gronchi in \cite{CG1}:
\begin{thm}[B\"{o}r\"{o}czky, \cite{B}]
For any convex body $K$ in $\mathbb{R}^n$ with $d_{\mathcal{BM}}(K,B)\geq 1+\varepsilon$ for some $\varepsilon >0$, there
exists an origin-symmetric convex body $C$ and a constant $\gamma'>0$ depending on $n$ such that $d_{\mathcal{BM}}(C,B)\geq 1+\gamma'\varepsilon^2$ and $C$ results from $K$ as a limit of subsequent Steiner symmetrization and linear transformations.
\end{thm}
\begin{remark}
In the statement of \cite[Theorem 1.4]{B}, it is mentioned that for an arbitrary convex body $K$, $C$ results from $K$ as a limit of subsequent Steiner symmetrization and affine transformations. However, no translation is needed. See Lemma 18 of \cite{Bor}.
\end{remark}
By means of shadow system, Campi and Gronchi proved the following theorem:
\begin{thm}[Campi and Gronchi, \cite{CG1}]
Let $K\in\mathcal{K}^n.$ The ratio $\frac{V(\Gamma K)}{V(K)}$ is non-increasing after a Steiner symmetrization applied to $K.$
\end{thm}
Now we give the proof in the general case. We prove by contraposition. Let $K\in \mathcal{K}^2$ be a convex body such that \[d_{\mathcal{BM}}(K,B)>1+\left(\frac{\gamma}{\gamma'}\right)^{\frac{1}{2}}\varepsilon^{\frac{1}{8}}.\] Then for an origin-symmetric convex body $C$, by the last two theorems, we have $\frac{V(\Gamma C)}{V(C)}\leq \frac{V(\Gamma K)}{V(K)}$ and $d_{\mathcal{BM}}(C,B)> 1+\gamma\varepsilon^{\frac{1}{4}}.$ Therefore,
\[\frac{V(\Gamma K)}{V(K)}\geq \frac{V(\Gamma C)}{V(C)}> \left(\frac{4}{3\pi}\right)^2(1+\varepsilon).\]
The proof of Theorem B is complete.
\section{higher order regularity}
In Section \ref{sec: Convergence in the Hausdorff metric} we proved for a sequence of times $\{t_k\}_k$ we have $\lim\limits_{t_k\to T}\frac{V(\Gamma K_{t_k})}{V(K_{t_k})}=\left(\frac{4}{3\pi}\right)^2.$ So by the monotonicity of $\frac{V(\Gamma K_t)}{V(K_t)}$, Corollary \ref{cor: monotonicity of petty-centroid}, we obtain
$\lim\limits_{t\to T}\frac{V(\Gamma K_t)}{V(K_t)}=\left(\frac{4}{3\pi}\right)^2.$ Hence, Theorem B implies that for each time $t\in[0,T)$ there exists a linear transformation $\Phi_t\in SL(2)$ such that
\begin{equation}\label{e: limit}
\lim_{t\to T}\frac{r_+(\Phi_tK_t)}{r_-(\Phi_tK_t)}=1.
\end{equation}
To obtain higher order regularity, we closely follow \cite{Ivaki4} with minor modifications. To this aim, (\ref{e: limit}) plays a basic role.

Proof of the next lemma is omitted, because of its similarity to the proof of Lemma 4.2 in \cite{Ivaki4} (see also \cite[Section 2]{Ivaki5}); the proof employs the method of Andrews introduced in \cite{BA1}.
\begin{lemma}[Harnack estimate]\label{lem: Harnack est} Under the evolution equation (\ref{e: asli}) we have
\[\partial_t\left(\frac{\mathcal{G}}{h^2}t^{\frac{1}{2}}\right)\geq 0,\mbox{~e.q.,~}\partial_t\left(\frac{\mathcal{G}}{h^2}\right)\geq -\frac{1}{2t}\left(\frac{\mathcal{G}}{h^2}\right).\]
\end{lemma}
\begin{lemma}\label{lem: app1} For each fixed $u\in\mathbb{S}^1$ and define
\[Q(t):=\frac{1}{2}(h_{K_t}(u)-h_{K_{t_0}}(u))+(t-t_0)\left(\frac{\mathcal{G}}{h^2}\right)(u,K_t),\]
on the time interval $\in[t_0,T)$.
Then $Q(t)\ge0$ on $[t_0,T).$
\end{lemma}
\begin{proof}
In light of Lemma \ref{lem: Harnack est}, after a time translation, we get
\[\partial_t\left(\frac{\mathcal{G}}{h^2}\right)\geq -\frac{1}{2(t-t_0)}\left(\frac{\mathcal{G}}{h^2}\right),~\mbox{for~all~}t>t_0.\]
So the time derivative of $Q$, given by $\partial_tQ=\frac{1}{2}\frac{\mathcal{G}}{h^2}+(t-t_0)\frac{\partial}{\partial t}\left(\frac{\mathcal{G}}{h^2}\right),$
is non-negative and moreover $Q(t_0)=0$.
\end{proof}
Next is an adjustment of the argument of Andrews and McCoy presented in Section 12 of \cite{BM}. We will obtain a lower bound on $\displaystyle\mathcal{G}/h^3$, the centro-affine curvature. The lower bound of the next lemma then conveniently provides a uniform lower bound for the Gauss curvature of the normalized solution. In what follows, a key property of $\displaystyle\mathcal{G}/h^3$ will repeatedly be used: For every $\Phi\in SL(2)$ and $K\in\mathcal{F}^2_0$, we have
\[\min_{u\in\mathbb{S}^1}\frac{\mathcal{G}}{h^3}(u,K)=\min_{u\in\mathbb{S}^1}\frac{\mathcal{G}}{h^3}(u,\Phi K)\mbox{~and~}
\max_{u\in\mathbb{S}^1}\frac{\mathcal{G}}{h^3}(u,K)=\max_{u\in\mathbb{S}^1}\frac{\mathcal{G}}{h^3}(u,\Phi K).\]
\begin{lemma}\label{lem: lower bound on the speed}
There exist an absolute constant $C>0$ and a time $t_{\ast}<T$, such that for each $t\geq t_{\ast}$ we have
\[\frac{\mathcal{G}}{h^3}(\cdot,t):=\frac{\mathcal{G}}{h^3}(\cdot,K_t)\geq \frac{C}{T-t}.\]
\end{lemma}
\begin{proof}
After a time shift, we may assume without loss generality that by (\ref{e: limit}) we have for a fixed $1\leq \eta< 1.5^{\frac{1}{4}}:$
\[\frac{r_{+}(\Phi_{\tau}K_\tau)}{r_{-}(\Phi_{\tau}K_\tau)}\leq \eta\]
for all $\tau\geq 0.$ Fix a $\tau\geq 0.$ Let $B_{r(t)}$ and $B_{R(t)}$ be solutions to flow (\ref{e: asli}) starting with $B_{r_{-}(\Phi_{\tau}K_{\tau})}$ and $B_{\eta r_{-}(\Phi_{\tau}K_{\tau})}$ respectively. Radii $R(t)$ and $r(t)$ are explicitly given by
\begin{equation}\label{eq: exp R}
R(t)=\left[\eta r_{-}(\Phi_{\tau}K_{\tau})^{4}-4(t-\tau)\right]^{\frac{1}{4}},
\end{equation}
\[r(t)=\left[r_{-}(\Phi_{\tau}K_{\tau})^{4}-4(t-\tau)\right]^{\frac{1}{4}},\]
and by the comparison principle
\[B_{r(t)} \subseteq \Phi_{\tau}K_t\subseteq B_{R(t)},\]
 for all $\tau \leq t\leq \tau+\frac{(r_{-}(\Phi_{\tau}K_{\tau}))^{4}}{4}.$ So we must have $T\geq\tau+\frac{(r_{-}(\Phi_{\tau}K_{\tau}))^{4}}{4}.$ Take $\tau^{\ast}:=\tau+\frac{(r_{-}(\Phi_{\tau}K_{\tau}))^{4}}{8}$ and an arbitrary $u\in\mathbb{S}^1$. By Lemma \ref{lem: app1} and equation (\ref{eq: exp R}) we obtain:
\begin{align*}
2\left[\eta^{4}-0.5\right]^{\frac{1}{4}}&r_{-}(\Phi_{\tau}K_{\tau})\frac{\mathcal{G}}{h^3}(u,\Phi_{\tau}K_{\tau^{\ast}})\\
&=2 R(\tau^{\ast})\frac{\mathcal{G}}{h^3}(u,\Phi_{\tau}K_{\tau^{\ast}})\\
&\geq\frac{2\mathcal{G}}{h^2}(u,\Phi_{\tau}K_{\tau^{\ast}})\\
&\geq \frac{h_{\Phi_{\tau}K_{\tau}}(u)-h_{\Phi_{\tau}K_{\tau^{\ast}}}(u)}{\tau^{\ast}-\tau}\\
&\geq \frac{8(r_{-}(\Phi_{\tau}K_{\tau})-R(\tau^{\ast}))}{(r_{-}(\Phi_{\tau}K_{\tau}))^{4}}\\
&=\frac{8\left(1-\left[\eta^{4}-0.5\right]^{\frac{1}{4}}\right)}{(r_{-}(\Phi_{\tau}K_{\tau}))^{3}}.
\end{align*}
Thus, we get
\[\frac{\mathcal{G}}{h^3}(u,\Phi_{\tau}K_{\tau^{\ast}})\geq \frac{8C}{(r_{-}(\Phi_{\tau}K_{\tau}))^{4}},~\mbox{for~all~}u\in\mathbb{S}^1\]
and some positive absolute constant $C$.
Recall that
\[T\geq\tau+\frac{(r_{-}(\Phi_{\tau}K_{\tau}))^{4}}{4}=\tau^{\ast}+\frac{(r_{-}(\Phi_{\tau}K_{\tau}))^{4}}{8}.\]
and $\mathcal{G}/h^3$ is invariant under $SL(2)$. Therefore, we conclude that
\[\frac{\mathcal{G}}{h^3}(u,\tau^{\ast})=\frac{\mathcal{G}}{h^3}(u,K_{\tau^{\ast}})\geq\frac{8C}{(r_{-}(\Phi_{\tau}K_{\tau}))^{4}}
\geq\frac{C}{T-\tau^{\ast}}.\]
Defining function $f$ on the time interval $[t_{\ast},T)$ by $f(\tau)=\tau+\frac{(r_{-}(\Phi_{\tau}K_{\tau}))^{4}}{8}-t,$
and using Proposition \ref{prop: area goes to zero} (which says $\lim\limits_{t\to T}V(\Phi_tK_t)=\lim\limits_{t\to T}V(K_t)=0$), it is not difficult to see that each $t\geq t_{\ast}:=\frac{(r_{-}(\Phi_{0}(K_{0})))^{4}}{8}$ can be written as $t=\tau+\frac{(r_{-}(\Phi_{\tau}K_{\tau}))^{4}}{8}$ for a $\tau\geq 0.$
\end{proof}
\begin{remark}
The comparison principle implies that for each $t\in[0,T):$
\begin{equation}\label{e: controling the radii}
\frac{(r_{-}(\Phi_tK_t))^{4}}{4}\leq T-t\leq \frac{(r_{+}(\Phi_tK_t))^{4}}{4}\leq \frac{(\eta r_{-}(\Phi_tK_t))^{4}}{4}.
\end{equation}
\end{remark}
\begin{lemma}\label{lem: upper bound on the speed}
There exists an absolute constant $C'>0$ such that on the time interval $T/2\leq t<T$ we have
\[\left(\frac{\mathcal{G}}{h^{3}}\right)(\cdot,K_t)\leq \frac{C'}{T-t}.\]
\end{lemma}
\begin{proof}
For each fixed $t^{\ast}\in [T/2,T),$ the family of convex bodies defined by $\tilde{K}^{t^{\ast}}_t:=\frac{1}{(T-t^{\ast})^{\frac{1}{4}}}\Phi_{2t^{\ast}-T}K_{t^{\ast}+(T-t^{\ast})t}$
is a solution of (\ref{e: asli}) on the time interval $[-1,0]$, with the initial data $\tilde{K}^{t^{\ast}}_{-1}=\frac{1}{(T-t^{\ast})^{\frac{1}{4}}}\Phi_{2t^{\ast}-T}K_{2t^{\ast}-T}$. By inequalities (\ref{e: controling the radii}), at the time $t=-1$, we have
\[r_{-}(\tilde{K}^{t^{\ast}}_{-1})=\frac{r_{-}(\Phi_{2t^{\ast}-T}K_{2t^{\ast}-T})}{(T-t^{\ast})^{\frac{1}{4}}}\geq \frac{8^{\frac{1}{4}}}{\eta}\] and
\[r_{+}(\tilde{K}^{t^{\ast}}_{-1})=\frac{r_{+}(\Phi_{2t^{\ast}-T}K_{2t^{\ast}-T})}{(T-t^{\ast})^{\frac{1}{4}}}\leq \eta8^{\frac{1}{4}}.\]
From the assumption $1\leq\eta<1.5^{\frac{1}{4}}$, and again by the comparison principle, for each time $t\in[-1,0]$, we get
\[r_{-}(\tilde{K}^{t^{\ast}}_{t})\geq \left(4\left(\frac{2}{\eta^{4}}-1\right)\right)^{\frac{1}{4}}\geq
\left(\frac{4}{3}\right)^{\frac{1}{4}},\]
and
\[r_{+}(\tilde{K}^{t^{\ast}}_{t})\leq \eta8^{\frac{1}{4}}< 12^{\frac{1}{4}}.\]
These two last inequalities combined with Lemma \ref{lem: upper G} imply that the centro-affine curvature of  $\frac{1}{(T-t^{\ast})^{1/4}}\Phi_{2t^{\ast}-T}K_{t^{\ast}}=\tilde{K}^{t^{\ast}}_{0}$ is bounded by a positive constant $C'$ independent of $t^{\ast}$. So
\[\frac{\mathcal{G}}{h^{3}}(\cdot,\Phi_{2t^{\ast}-T}K_{t^{\ast}})\leq \frac{C'}{T-t^{\ast}}\Rightarrow \frac{\mathcal{G}}{h^{3}}(\cdot,K_{t^{\ast}})\leq \frac{C'}{T-t^{\ast}}.\]
Taking into account that $t^{\ast}\in [T/2,T)$ is arbitrary and $C'$ is independent of $t^{\ast}$ completes the proof.
\end{proof}
\subsection{Proof of Theorem A: $\mathcal{C}^{\infty}$ convergence of the normalized solution}\label{sec: main theorem proof}
For a fixed $t^{\ast}\in \left[\max\{3T/4,\frac{T+t_{\ast}}{2}\},T\right),$ the family of convex bodies given by $\tilde{K}^{t^{\ast}}_t=\frac{1}{(T-t^{\ast})^{\frac{1}{4}}}\Phi_{2t^{\ast}-T}K_{t^{\ast}+(T-t^{\ast})t}$
is a solution of (\ref{e: asli}) on the time interval $[-1,0]$ with the initial data $\tilde{K}^{t^{\ast}}_{-1}=\frac{1}{(T-t^{\ast})^{\frac{1}{4}}}\Phi_{2t^{\ast}-T}K_{2t^{\ast}-T}$ and with the properties
\[r_{-}(\tilde{K}^{t^{\ast}}_{t})\geq
\left(\frac{4}{3}\right)^{\frac{1}{4}},\]
and
\[r_{+}(\tilde{K}^{t^{\ast}}_{t})< (12)^{\frac{1}{4}}.\]
Since $2t^{\ast}-T\geq \max\{T/2,t_{\ast}\}$, by Lemmas \ref{lem: lower bound on the speed} and \ref{lem: upper bound on the speed} we get
\[\frac{C}{2(T-t^{\ast})}\leq \frac{\mathcal{G}}{h^3}(\cdot,2t^{\ast}-T)\leq \frac{C'}{2(T-t^{\ast})}.\]
Thus, the centro-affine curvature of $\tilde{K}^{t^{\ast}}_{-1}$ satisfies
\[\frac{C}{2}\leq \frac{\mathcal{G}}{h^3}(\cdot,\tilde{K}^{t^{\ast}}_{-1})\leq \frac{C'}{2}.\]
To prove Theorem A, we mention two basic observations contained in the proofs of Corollary \ref{cor: Convexity is preserved} and Lemma \ref{lem: upper bound for speed}.
\begin{enumerate}
  \item $\min\limits_{\mathbb{S}^{1}} \mathcal{G}/h^2(\cdot,t)$ is non-decreasing along the flow.
  \item $\max\limits_{\mathbb{S}^{1}} \mathcal{G}/h^2(\cdot,t)$ remain bounded from above provide that the inner radius has a lower bound. This follows from Tso's trick which is applying the maximum principle to the auxiliary function $\displaystyle\Omega(u,t):=\frac{\frac{1}{h^2\mathcal{S}}(u,\tilde{K}^{t^{\ast}}_t)}{h(u,\tilde{K}^{t^{\ast}}_t)-\frac{1}{2}\left(\frac{4}{3}\right)^{\frac{1}{4}}}$. Furthermore, the upper bound on the speed depends only on $\max\limits_{\mathbb{S}^1}\Phi(\cdot,0)$, outer radius of the initial convex body, and the lower bound on the inner radii of the evolving convex bodies.
\end{enumerate}
Using observations (1) and (2), we conclude, for $t\in[-1,0]$, that
\[C_1\leq \frac{\mathcal{G}}{h^2}(\cdot,\tilde{K}^{t^{\ast}}_t)\leq C_2.\]
For constants $C_1$ and $C_2$ independent of $t^{\ast}.$ These constants are independent of $t^{\ast}$ as they only depend only on $C$, and $C'$.
Hence, each body $\tilde{K}^{t^{\ast}}_t$, for $t\in[-1,0]$, satisfies
\[C_3\leq \mathcal{G}(\cdot,\tilde{K}^{t^{\ast}}_t)\leq C_4\]
for some constants $C_3$ and $C_4$ independent of $t^{\ast}.$ Therefore, by \cite{K} we conclude that there are uniform bounds on higher derivatives of the curvature of $\tilde{K}^{t^{\ast}}_t$, for $t\in[-1,0].$ In particular, the body
\[\tilde{K}^{t^{\ast}}_0=\frac{1}{(T-t^{\ast})^{\frac{1}{4}}}\Phi_{2t^{\ast}-T}K_{t^{\ast}}\]
has uniform $\mathcal{C}^{k}$ bounds independent of $t^{\ast}.$

Consequently, for every given sequence of times $\{t_k\}_k$ we can find a subsequence, denoted again by $\{t_k\}_k$, such that as $t_k\to T$ the family $\left\{\frac{1}{(T-t_k)^{\frac{1}{4}}}\Phi_{2t_k-T}K_{t_k}\right\}$ approaches in $\mathcal{C}^{\infty}$ to a smooth convex body $K_{\infty}\in\mathcal{K}^2_e$.
Furthermore, we may use the monotonicity of $V(\Gamma K_{t})/V(K_{t})$ and the discussion in Section \ref{sec: Convergence in the Hausdorff metric} to conclude that
$V(\Gamma K_{\infty})/V(K_{\infty})=(4/3\pi)^2$. That is, $K_{\infty}$ is a smooth minimizer of $V(\Gamma K)/V(K)$. So by Corollary \ref{cor: monotonicity of petty-centroid} or Theorem B, $K_{\infty}$ is an origin-centered ellipse. Finally, notice that $V(K_t)$ is comparable with $\frac{1}{(T-t)^{1/2}}$ which results in the convergence of $\sqrt{\frac{V(B)}{V(K_{t_k})}}\Phi_{2t_k-T}K_{t_k}$ in $\mathcal{C}^{\infty}$ to an origin-centered ellipse.
\begin{remark}
Let $K_0\in \mathcal{F}_{0}^n$ be a convex body with its centroid at the origin and let the family of convex bodies $\{K_t\}\in \mathcal{F}_{0}^n$ be the solution to
\begin{equation}\label{e: badali}
\left\{
  \begin{array}{ll}
    \partial_t h(u,t)=-\frac{1}{h^n\mathcal{S}}(u,t),\\
    h(\cdot,0)=h_{K_0}(\cdot).
  \end{array}
\right.
\end{equation}
Evolution equation (\ref{e: badali}) is a fully nonlinear degenerate second-order parabolic differential equation with a high degree of homogeneity. In general, dealing with such flows is technically difficult, \cite{RS,AS,BA3,BM,BX,BX1,Ch,Ch1,G1,PL,QLI,Shu1,Shu,FS,WTL}. It would be interesting to study the asymptotic behavior of (\ref{e: badali}) via the evolution equations of affinely associated bodies such as $K^{\ast}$, $\Gamma K$, and $\Lambda K$.
\end{remark}
\bibliographystyle{amsplain}

\end{document}